\documentclass[12pt]{amsart}

\usepackage{amsmath,amssymb}
\usepackage{graphicx}
\usepackage{amsthm}
\usepackage[colorlinks=true]{hyperref}

\newcommand{\R}{\mathbb{R}}

\newcommand{\Z}{\mathbb{Z}}
\newcommand{\C}{\mathbb{C}}
\newcommand{\Q}{\mathbb{Q}}

\newcommand{\beqnn}{\begin{eqnarray*}}
\newcommand{\eeqnn}{\end{eqnarray*}}
\newcommand{\beqn}{\begin{eqnarray}}
\newcommand{\eeqn}{\end{eqnarray}}
\newcommand{\beq}{\begin{equation}}
\newcommand{\eeq}{\end{equation}}

\theoremstyle{plain}
\newtheorem{thm}{Theorem}[section]
\newtheorem{prop}[thm]{Proposition}

\newtheorem{cor}[thm]{Corollary}
\newtheorem{rmk}[thm]{Remark}
\newtheorem{defi}{Definition}[section]

\begin{document}

\title[Masur's Divergence for Tori and Kummer Surfaces]{Masur's Divergence for Higher Dimensional Tori and Kummer Surfaces}

\author{Zhijing Wang}
\address{Yau Mathematical Sciences Centre\\ Tsinghua University\\ Beijing, China}
\email{zj-wang19@mails.tsinghua.edu.cn}

\begin{abstract}
Masur's divergence states that the horizontal foliation of translation surfaces is uniquely ergodic if the geodesic flow is recurrent on the moduli space. This established a relationship between geometrical properties of foliations and the dynamics on the moduli space. In this paper, we extend this theorem to complex torus and Kummer surfaces. We define and calculate horizontal foliations and the corresponding geodesic flow in the moduli space of K\"ahler metrics and prove that the horizontal foliation is uniquely ergodic if the geodesic flow is recurrent. We also find the a necessary algebraic condition on the geodesic flow for the horizontal foliation to be uniquely ergodic. \end{abstract}

\maketitle
\tableofcontents
\section{Introduction}
\subsection{Motivations}
Given a surface $S_g$ with genus $g> 1$, its Teichm\"uller space $\mathrm{Teich}(S_g)$ is defined as complex structures on $S_g$ up to isotopy, which is equivalent to the space of hyperbolic metrics on $S_g$ up to isotopy. The Teichm\"uller distance on the Teichm\"uller space is given by the logarithmic of the infimum dilatation among maps between the two complex structures: $d_{Teich}(X,Y)=\inf_{f:X\to Y}\log K_f$ where the dilatation $K_f$ of a map $f$ is given by $K_f=\sup_{p\in X}\frac{|f_z|+|f_{\bar{z}}|}{ |f_z|-|f_{\bar{z}}|}$ which at every point describes the ratio of the major and minor axes of the ellipse given by the image of infinitesimal circles under $f$. The Teichm\"uller distance roughly measures how much two conformal structures are different. 

Teichm\"uller showed that for every diffeomorphism on $S_g$, there is a unique map (differentiable outside finitely many points) in its isotopy class, called the Teichm\"uller map, minimizing the dilatation, thus achieving the Teichm\"uller distance. 

The Teichm\"uller map is described explicitly by holomorphic quadratic differentials. For a Teichm\"uller map $f: X\to Y$, there are two holomorphic quadratic differentials $q$ on $X$ and $p$ on $Y$, such that under the natural coordinates of $q$ and $p$ that writes $q=dz^2$ and $p=dw^2$, $f$ is of the form $\Re w=\sqrt{K_f}\Re z, \Im w=\frac{1}{\sqrt{K_f}}\Im z$. Given a quadratic differential $q$ on $X$, taking the horizontal foliations and vertical foliations to be parallel to the $x$ and $y$ axes in the natural coordinates of $q$ and $p$, we see that the Teichm\"uller map simply contracts the vertical foliations and expands the horizontal foliations. 

The Teichm\"uller flow is given as the geodesic flow on the Teichm\"uller space under the Teichm\"uller distance. We also call it the geodesic flow on Teichm\"uller space and denote it by $Ra_t(X,q)$ where the quadratic differential $q$ is identified with the cotangent space of $\mathrm{Teich}(S_g)$. The Teichm\"uller flow gives a set of complete geodesics on the Teichm\"uller space, whose end points can be identified with the horizontal and vertical foliations of the quadratic differentials.  

The moduli space of $S_g$ is defined as the complex structures on $S_g$ up to diffeomorphisms, so it is a quotient of the Teichm\"uller space by the mapping class group $\Gamma=\mathrm{Diff}(S_g)/\mathrm{Diff}_0(S_g)$ where $\mathrm{Diff}_0(S_g)$ is the group of diffeomorphisms on $S_g$ and $\mathrm{Diff}_0(S_g)$ is the identity component of the diffeomorphism group, or diffeomorhpisms of $S_g$ isotopic to identity. The mapping class group acts properly discontinuously on the Teichm\"uller space, so the geodesic flows descends to the moduli space.

Masur's divergence established a relationship between unique ergodicity of foliations and the dynamics of the geodesic flow on the moduli space.

\begin{thm}[\cite{Masdiv}] Let $X$ be a Riemann surface with quadratic differential $q$. Suppose the Teichm\"uller flow $Ra_t(X,q)$ is recurrent in the moduli space, then the horizontal foliation of $(X,q)$ is uniquely ergodic.
\end{thm}

Recently, there have been many interesting results making analogies between Riemann surfaces and higher dimensional objects. For example, Cantat \cite{Canaut} established an analogy of stable and unstable foliations for pseudo-Anosov mappings, proving that that there exists stable and unstable currents contracted and expanded by the hyperbolic automorphisms of K3 surface. Filip \cite{filcount} made an analogy between periodic billiard trajectories on translation surfaces and special Lagrangian tori on K3 surfaces and counted the growth of the number of special Lagrangian tori in a twistor family.

\subsection{Main results}
In this article, we are interested in extending the definitions and properties of foliations and the Teichm\"uller flow to higher dimensional subjects. We start with the case of complex tori. 

Unlike the case of Riemann surfaces, for higher dimensional objects, the space of complex structures up to isotopy and the space of flat (hyperbolic) metrics up to isotopy are different objects. For the complex torus, and hyperk\"ahler manifolds (including K3 surfaces), the moduli space of complex structures is even highly non-Hausdorff. In fact, almost all complex structures are ergodic in the moduli space \cite{Vererg}. Thus to study dynamics on the moduli space and to achieve geometric results related to the dynamics, the non-Hausdorff complex moduli space may not be the appropriate analogy.

In this article, we consider instead the moduli space of K\"ahler metrics on a differentiable manifold $M$ which is the set of equivalent K\"ahler structures of $M$ up to isotopy. We define this as $$\mathrm{Kah}(M) = \{(X,I,\omega) \} / \sim$$ where $(X,I,\omega)$ are taken over K\"ahler manifolds $X$ equipped with complex structure $I$ and K\"ahler form $ \omega$ such that $X$ is diffeomorphic to $M$; the equivalence relation is defined as $(X,I,\omega)\sim (Y,J,\nu)$ if there exists a homeomorphism $f:X\to Y$ such that $f^*J=I$ and $[f^*\nu]=[\omega]\in H^2(X,\R)$.

For a complex torus, up to an isometry of K\"ahler structures, we may view the complex torus as $\C^n/\Lambda$ for a lattice $\Lambda$ of $\C^n$ with the standard K\"ahler metric $\omega_0=\sum_{i=1}^n dx_i\wedge dy_i$ on $\C^n$. The moduli space of K\"ahler metrics of a $2n$-torus $\mathbb{T}^{2n}$ is $\mathrm{Kah}(\mathbb{T}^{2n})=\mathrm{SL}(2n,\Z)\backslash \mathrm{SL}(2n,\R)/\mathrm{U}(n)$.

The group $\mathrm{SL}(2n,\R)$ acts on $\mathrm{SL}(2n,\Z)\backslash \mathrm{SL}(2n,\R)$ by right multiplication. This descends to the moduli space of K\"ahler structures of complex torus. More concretely, for a torus  $X=\C^n/\Lambda(g)$, we identify the lattice $\Lambda(g)$ with a matrix  $g\in \mathrm{SL}(2n,\R)$ by taking a basis of $\Lambda(g)$ to be the row vectors of $g$ and identifying $\C^n$ with $\R^{2n}$ by $(z_1=x_1+iy_2,...,z_n=x_n+iy_n)\mapsto(x_1,y_1,...,x_n,y_n)$.

We define the geodesic flow $Ra_t$ on the moduli space to be image of the action of $a_t=\mathrm{diag}(e^{-t}, e^{t},...,e^{-t}, e^{t})$ on $\mathrm{SL}(2n,\Z)\backslash \mathrm{SL}(2n,\R)$ in the moduli space $\mathrm{Kah}(\mathbb{T}^{2n})$, given as $Ra_t(\C^n/\Lambda(g))=\C^n/\Lambda(ga_t)$. The geodesic flow is said to be \emph{recurrent} if it visits a compact subset of the moduli space infinitely many times.

We define the \emph{horizontal} and \emph{vertical foliations} on $\C^n/\Lambda$ to be $n$-dimensional foliations on the torus whose leaves are the subspaces in $\C^n$ parallel to the $x$-axis and $y$-axis respectively. A foliation is said to be \emph{uniquely ergodic} if it only has one transverse measure up to a non-zero scalar multiplication. This is roughly saying that the foliation is equally distributed on the manifold.
In this way, the geodesic flow acts on the torus $\C^n/\Lambda$ by contracting the horizontal foliations and expanding the vertical foliations. 

Then we have the following analogy for Masur's criterion.

\begin{thm}\label{masurtorus}Let $X=\C^n/\Lambda(g)$ be a $2n$-torus for $g\in GL(2,\R)$. Assume that the geodesic flow $\{Ra_t X: t\le 0\}$ is recurrent in the moduli space of K\"ahler metrics $\mathrm{Kah}(\mathbb{T}^{2n})$. Then the horizontal foliation on $X$ must be uniquely ergodic.
\end{thm}

This illustrates an interesting relationship between dynamical property of the geodesic flow on the geometric property of foliations on $X$.

We remark that there are, of course, many ways to define geodesic flows on $\mathrm{SL}(2n,\R)$, and we choose $a_t=\mathrm{diag}(e^{-t}, e^{t},...,e^{-t}, e^{t})$ so that it acts on foliations similar to the way Teichm\"uller flow does for complex dimension 1. In fact, any action with exactly $n$ eigenvalues larger than 1 and $n$ eigenvalues smaller than 1 would have similar properties, one simply needs to define the horizontal and vertical foliations to be the respective eigen-spaces.

For real dimension $2$ where $X=\mathbb{T}^2$, Masur's divergence is simply given by the fact that non-rational foliations are uniquely ergodic. (If a foliation is not uniquely ergodic, then a geodesic curve of $X$ must be a leaf of the foliation, so the length of the geodesic tends to zero along the geodesic flow $Ra_t X$ which implies that $Ra_t X$ must be divergent. ) For higher dimensions, the failure of unique ergodicity of the horizontal foliation does not imply the existence of a closed leaf of the foliation, so Theorem \ref{masurtorus} cannot be simply deduced by the fact (See Theorem \ref{slftorus}) that non-rational straight line flows on torus are uniquely ergodic. 

Furthermore, the converse of Theorem \ref{masurtorus} is not true for higher dimensional complex tori, as is the case for Riemann manifolds. For example, on the torus $\C^2/\Lambda(g)$ where $\Lambda(g)$ is a lattice in $\C^2$ spanned by the row vectors of 
$g=\begin{pmatrix} 1 &0\\ i&0\\ \sqrt{3} +\sqrt{5}i & 1+\sqrt{2}i  \\ 0&i\end{pmatrix}$, where $i=\sqrt{-1}$, the horizontal foliation is uniquely ergodic but the geodesic flow starting at $g$ is divergent. For proof see Proposition \ref{exp}.

We next extend the result to other complex manifolds. The simplest extension is to Kummer surfaces. A Kummer surface is defined as the minimal resolution of a complex torus $X$ mod involution. More explicitly, taking $\tau$ to be the involution on $\C^2/\Lambda(g)$ descending from the inversion map $(z_1,z_2)\to (-z_1,-z_2)$ on $\C^2$, the Kummer surface $\mathrm{Kum}(g)$ is given by the blow-up of $(\C^2/\Lambda(g))/\tau$ at the 16 fixed points of $\tau$. We denote by $C_j, j=1,...,16$ the $16$ exceptional curves given by the blow-ups of the fixed points. The orbifold $(\C^2/\Lambda(g))/\tau$ before blow-up is called the exceptional Kummer surface. Kummer surfaces are examples of K3 surfaces. 

For Kummer surfaces, we consider the moduli space of degenerate K\"ahler forms, i.e. K\"ahler forms defined outside some exceptional curves, denoted by  $\overline{\mathrm{Kah}}(M_0)$ (see Section \ref{kum} for more details on this) where $M_0$ is the differentiable manifold underlying Kummer surfaces. Moreover, we define the geodesic flow in the moduli space of Kummer surfaces to be given by that of the torus, i.e. $Ra_t \mathrm{Kum}(g)=\mathrm{Kum}(ga_t)$. We define the horizontal foliation and vertical foliation of a Kummer surface to be the image of the horizontal and vertical foliations of the torus. 

Thus using our result for torus, we also establish the analogy of Masur's divergence to Kummer surfaces.

\begin{thm}\label{masurkum}Let $X$ be a Kummer surface. Assume that the geodesic flow $\{a_t X: t\le 0\}$ is recurrent in the moduli space of degenerate K\"ahler metrics $\overline{\mathrm{Kah}}(X)$. Then the horizontal foliation on $X$ must be uniquely ergodic.
\end{thm}

The moduli space of degenerate K\"ahler metrics of a Kummer surface, which is the same as the moduli space of degenerate K\"ahler metrics of K3 surfaces, can be explicitly calculated via the period map, using the Torelli theorem.


We also have the following necessary condition for a horizontal foliation on the Kummer surface to be uniquely ergodic. 

\begin{thm}\label{rank} For a Kummer surface $X$ with flat K\"ahler form $\omega$, let $\mathcal{F}^h$ be its horizontal foliation. Then there exists cohomology classes $\eta_1,\eta_2,\eta_3\in H^2(X,\R)$ such that the geodesic flow $Ra_t(X)$ is of the form $[e^{-2t}\eta_1+e^{2t}\eta_2+\sqrt{-1}\eta_3,\omega]\in  \mathrm{\overline{Kah}}(M_0)$. Moreover, $\mathcal{F}^h$ is uniquely ergodic only if $rank(\eta_2^\perp \cap H^2(X,\Z))< 19$. 
\end{thm}

We remark that the cohomology classes $\eta_1$, $\eta_2$ are on the boundary of the period domain and our result suggests that they can be thought of as representing the vertical and horizontal foliations. This gives an analogy of the identification of the boundary of Teichm\"uller space with horizontal and vertical foliations.

The form of the geodesic flow seems to suggest a generalization to general K3 surfaces. However, in \cite{McmSie} McMullen constructed examples of hyperbolic automorphisms of K3 surfaces that admits Siegel disks. The work of Cantat and Dupont \cite{CanKum} (for the projective case) and Filip and Tosatti \cite{Filkum} (for the general case)  on Kummer rigidity also suggests that foliation preserved by automorphisms of K3 surfaces cannot have full Lebesgue measure except for the case of Kummer surfaces we discussed in this article. Thus, it may be too optimistic to hope for horizontal and vertical foliations on general K3, even outside divisors. Nevertheless, these results do not rule out the possibility of geodesic flows not containing automorphisms to have full measure horizontal and vertical foliations. Moreover, if we hope the geodesic flows to contain automorphism as in translation surfaces, there is still hope of developing the theory for laminations instead of foliations to be preserved by ``geodesic flows''. 

\subsection{Organization}
In Section \ref{moduli} we recall the preliminaries on moduli spaces. In Section \ref{geodesic} we define the geodesic flow and foliations for complex torus. In Section \ref{versus} we prove Theorem \ref{masurtorus} for torus and discuss the converse. In Section \ref{kum}, we briefly recall the definitions and properties of K3 surfaces and Kummer surfaces, especially the period map and the generalized moduli spaces. In Section \ref{slag}, we apply the results on torus to Kummer surfaces and calculate explicit form and properties of the geodesic flow.

\subsubsection*{Acknowledgements}
I would like to thank my instructor Jinxin Xue and Yitwah Cheung for introducing me to the topic, for useful discussions and for comments on this article. 
\section{Moduli Spaces}\label{moduli}
First we give the definitions for the moduli spaces we shall be interested in, namely, the moduli space of complex structures and the moduli space of K\"ahler structures. 
\begin{defi}
For a differentiable manifold $M$, we denote by $\mathrm{Comp}(M)$ the \emph{space of complex structures} on M, i.e. $$\mathrm{Comp}(M)=\{J\in \mathrm{End}(TM), J^2=-1, [T^{1,0}_JM,T^{1,0}_JM]\subset T^{1,0}_JM\}.$$

We define its \emph{Teichm\"uller space} to be the space of complex structures up to isotopy: $$\mathrm{Teich}(M)=\mathrm{Comp}(M)/\mathrm{Diff}_0(M).$$

The \emph{mapping class group} of $M$ is defined as the diffeomorphisms of $M$ up to isotopy: $$\Gamma(M)=\mathrm{Diff}(M)/\mathrm{Diff}_0(M).$$

The \emph{moduli space of complex structures} of $M$ is the space of complex structures up to diffeomorphisms $$\mathrm{Mod}(M)=\mathrm{Comp}(M)/\mathrm{Diff}(M)=\mathrm{Teich}(M)/\Gamma(M).$$

The \emph{moduli space of K\"ahler structures} of $M$ is defined as $$\mathrm{Kah}(M) = \{(X,I,\omega) \} / \sim$$ where $(X,I,\omega) $ are taken over a K\"ahler manifold $X$ equipped with complex structure $I$ and K\"ahler form $ \omega$ which is diffeomorphic to $M$ and the equivalence is defined as $(X,I,\omega)\sim (Y,J,\nu)$ if there exists a homeomorphism $f:X\to Y$ such that $f^*J=I$ and $[f^*\nu]=[\omega]\in \mathrm{P}H^2(X,\R)$.
\end{defi}

\begin{rmk} For complex tori, the moduli space of K\"ahler structures is exactly the set of flat K\"ahler structures of volume 1 on $M$. For Calabi-Yau manifolds, this is the set of Ricci-flat K\"ahler structures of volume 1 on $M$ as proved by Yau's celebrated theorem that every K\"ahler class of a Calabi-Yau manifold admits a unique Ricci-flat K\"ahler metric \cite{Yau}.
\end{rmk}

For the complex tori, the calculation of Teichm\"uller space, the mapping class group and the moduli spaces is a classical result. 

We calculate that the Teichm\"uller space of an $2n$-torus is $\mathrm{Teich}(\mathbb{T}^{2n})=\mathrm{GL}(2n,\R)/\mathrm{GL}(n,\C)$. Indeed, for every complex structure $J$ on $X\simeq \mathbb{T}^{2n}$, we may lift the complex structure  to a complex structure $\tilde{J}$ on $\R^{2n}$ which is invariant under translations by the fundamental group $\pi_1(X)$. Thus there exists an isotopy on $\R^{2n}$ commuting with $\pi_1(X)$ (so it descends to an isotopy on $\mathbb{T}^{2n}$) that takes $\tilde{J}$ to the standard complex structure $J_0$ on $\C^n$. This identifies $X,J$ with $(\C^n,J_0)/\pi_1(X)$ in the Teichm\"uller space $\mathrm{Teich}(\mathbb{T}^{2n})$. The lifting of the $2n$ generators of $\pi_1(X)$ in $\C^n$ may be expressed as an element of $\mathrm{GL}(2n,\R)$. These generators determine an element of the Teichm\"uller space of a $2n$-torus up to complex linear transformation. Thus we have $\mathrm{Teich}(\mathbb{T}^{2n})=\mathrm{GL}(2n,\R)/\mathrm{GL}(n,\C)$.

Diffeomorphisms of the torus up to isotopy are given exactly by automoprhism of the fundamental group, so the mapping class group of an $2n$-torus is $\Gamma(\mathbb{T}^n)=\mathrm{SL}(2n,\Z)$. 

Thus the moduli space of complex structures is given by $\mathrm{Mod}(\mathbb{T}^{2n})=\mathrm{Teich}(\mathbb{T}^{2n})/\Gamma=\mathrm{SL}(2n,\Z)\backslash \mathrm{GL}(2n,\R)/\mathrm{GL}(n,\C)$. 

 By similar arguments, the K\"ahler structures are always equivalent to the standard K\"ahler structure on $(\C^n/\Lambda,J_0,\omega_0)$ where $\Lambda$ is a lattice in $\C^n$. And two lattices induce the same K\"ahler structure if and only if they differ by a unitary transformation and a rescaling. Thus the moduli space of flat K\"ahler structures is $$\mathrm{Kah}(\mathbb{T}^{2n})=\mathrm{SL}(2n,\Z)\backslash \mathrm{GL}(2n,\R)/(\mathrm{U}(n)\times \R)\simeq \mathrm{SL}(2n,\Z)\backslash \mathrm{SL}(2n,\R)/\mathrm{U}(n).$$

We note that the complex moduli space of $\mathbb{T}^{2n}$ is non-Hausdorff for $n>1$ since $\mathrm{GL}(n,\C)$ acts ergodically on $\mathrm{SL}(2n,\Z)\backslash \mathrm{GL}(2n,\R)$. The flat K\"ahler moduli space is Hausdorff since $\mathrm{U}(n)$ is compact.

\section{Geodesic Flow and Foliations on Complex Torus}\label{geodesic}
As for the case of Riemann surfaces where we considered the natural coordinates, for higher dimensional object we may also define the geodesic flow and the horizontal and vertical foliations using the flat structures. In the case of complex tori, the flat structure readily comes from the universal covering $\C^n$.

\subsection{Foliations}

First we briefly introduce the definition of measured foliations.
\begin{defi}
A $n$ dimensional \emph{foliation} $\mathcal{F}$ on a $2n$-dimensional differentiable manifold $M$ is defined as a decomposition of $M$ into a collection of disjoint union of subsets of $M$, called the leaves of $\mathcal{F}$, such that for each $p\in M$ there is a smooth chart from a neighborhood of $p$ to $\C^n$ that takes the leaves to the horizontal planes $(n$-planes parallel to the $x$-axes$)$ with transition maps of the form $(x_1,y_1,...,x_n,y_n)\to (f_1(x_1,y_1,..,x_n,y_n),g_1(y_1,...,y_n),....,$$f_n(x_1,y_1,..,x_n,y_n),g_n(y_1,...,y_n))$. 

A \emph{singular foliation} is a foliation defined outside finitely many points.

A \emph{transverse measure} $\mu$ of a $($singular$)$ foliation $\mathcal{F}$ is a map that assigns a non-negative real number to every smooth $n$-submanifold of $M$ such that $\mu$ is invariant under isotopies of $M$ that preserves each leaf of $\mathcal{F}$ and $\mu$ is absolutely continuous with respect to the $n$-dimensional Lebesgue measure. 

A \emph{measured ($n$-)foliation} $(\mathcal{F},\mu)$ on a $2n$-manifold $M$ is a foliation $\mathcal{F}$ equipped with a transverse measure $\mu$.

A foliation is said to be \emph{uniquely ergodic} if it has only one transverse measure up to scale.
\end{defi}

We shall be interested in very nice foliations, namely the horizontal and vertical foliations on $\mathbb{T}^{2n}$, they are Lagrangian, or even special Lagrangian foliations. First we briefly introduce the definitions (See section 7, part 1 of \cite{CY} for detailed definition and examples on this). 
\begin{defi}
A foliation $\mathcal{F}$ on a K\"ahler manifold $(X,\omega)$ is called \emph{Lagrangian} if the K\"ahler form vanishes when restricted to each leaf $F$, i.e. $\omega|_F=0$. 

In the case that $X$ is equipped with an holomorphic  $n$-form $\Omega$, $($e.g. $X$ is a Calabi-Yau or hyperk\"ahler manifold$)$, we say that a foliation $\mathcal{F}$ is \emph{special Lagrangian} if it is Lagrangian, the imaginary part of $\Omega$ vanishes  and the real part of $\Omega$ is positive when restricted to each leaf $F$, i.e. $\omega|_F=0,\Im(\Omega)|_F=0,\Re(\Omega)|_F>0$ for each leaf $F$.
\end{defi}
Now we define the horizontal and vertical foliations on complex tori in the following way. This is motivated by the definition of horizontal and vertical foliations on Riemann surfaces which is given by the natural coordinates. 

We consider the universal covering of the torus.
For $g=(g_{ij})_{1\le i,j\le 2n}\in \mathrm{SL}(2n,\R)$ we take $\Lambda(g)$ to be the lattice in $\C^n$ given by the row vectors of 
\begin{equation}\label{matrix}\begin{pmatrix} g_{11}+g_{12}i & g_{13}+g_{14} i&...&g_{1,2n-1}+g_{1,2n}i\\ g_{21}+g_{22}i & g_{23}+g_{24}i &...&g_{2,2n-1}+g_{2,2n}i\\\vdots& \vdots&\ddots &\vdots\\ g_{2n,1}+g_{2n,2} i& g_{2n,3}+g_{2n,4} i&...&g_{2n,2n-1}+g_{2n,2n} i\end{pmatrix}
\end{equation} where $i=\sqrt{-1}$. We take $\C^n\to \C^n/\Lambda(g)=X$ the universal covering of the complex torus $X$ given by $g$. The leaves of the horizontal foliation $\mathcal{F}^h$ on $X$ given by $g$ are the images of $\{(x_1+iy_1,...,x_n+iy_n)\in\C^2: y_1=c_1,...,y_n=c_n\}\subset \C^n$ in $X$ and the transverse measure is given by $d\mu_h=dy_1\wedge...\wedge dy_n$. Respectively, the vertical foliation $\mathcal{F}^v$ is given by $\{(x_1+iy_1,...,x_n+iy_n)\in\C^2: x_1=c_1,...,x_n=c_n\}$ in $X$ with transverse measure $d\mu_v=dx_1\wedge...\wedge dx_n$. The horizontal and vertical foliations are transverse Lagrangian foliations on $X$, where the K\"ahler form is given by $\omega_0$ where $\omega_0=\sum dx_i\wedge dy_i$ is the standard one on $\C^n$.

We note that the horizontal and vertical foliations are taken with respect to $g\in\mathrm{SL}(2n,\R)$. If we have a complex torus $X$ with a flat K\"ahler metric $\omega$ to begin with and require that $(X,\omega)=(\C^n/\Lambda(g),\omega_0)$, then there is a $\mathrm{U}(n)$ set of possible choices for $g$ and a $\mathrm{U}(n)/\mathrm{O}(n)$ set of choices for the horizontal and vertical foliations. In fact, in the case $n=2$, every linear special Lagrangian foliation in a flat K\"ahler torus can be given as the horizontal foliation for some choice of $g\in \mathrm{SL}(2n,\R)$ (c.f. Proposition \ref{lagfol}).

\subsection{Geodesic flow}
The geodesic flow is given as expansion of the vertical foliations and contraction of the horizontal foliations. This is motivated by the action of Teichm\"uller flow on natural coordinates. We may also write this more explicitly in the language of a homogeneous flow. 

\begin{defi}
We denote by $Ra_t $ the \emph{geodesic flow} on the homogeneous space $\mathrm{SL}(2n,\Z)\backslash \mathrm{SL}(2n,\R)$ starting at $y\in \mathrm{SL}(2n,\Z)\backslash \mathrm{SL}(2n,\R)$ given by $Ra_t(y)= ya_t$ where $$a_t=diag(e^{-t}, e^{t},...,e^{-t}, e^{t})$$ is the diagonal matrix of entries $e^{-t}$ and $e^t$ on the real and imaginary parts respectively.  

For a complex torus $X$ expressed as $\C^n/\Lambda(g)$ for $g\in \mathrm{SL}(2n,\R)$, the geodesic flow (depending on $g$) is defined as $$Ra_t(X,\omega)=(\C^n/\Lambda(Ra_tg),\omega_0)=(\C^n/\Lambda(ga_t), \omega_0)$$ 
\end{defi}

Take $X=\C^n/\Lambda(g)$ and consider the horizontal and vertical foliations of $X$ given by $\C^n\to X$, we see that the geodesic flow $R_{a_t}$ preserves the horizontal and vertical foliations and while contracting and expanding the transverse measures. 

We may project the geodesic flow down to the moduli spaces of the torus by projecting $\mathrm{SL}(2n,\Z)\backslash \mathrm{SL}(2n,\R)$ down to the K\"ahler moduli space $\mathrm{Kah}(\mathbb{T}^{2n})\simeq \mathrm{SL}(2n,\Z)\backslash \mathrm{SL}(2n,\R)/\mathrm{U}(n)$, the K\"ahler moduli space of complex torus. More explicitly, this is given by the right multiplication $Ra_t((\C^n,\omega_0)/\Lambda(g))= (\C^n,\omega_0)/\Lambda(ga_t)$, where $\omega_0$ is the standard K\"ahler form on $\C^n$.

\section{Unique Ergodicity versus Recurrence}\label{versus}
Now we are ready to prove our analogy of Masur's divergence on the case of complex torus. The rough idea is to approximate the transverse measure of a transversal sub-manifold with the number of intersections of the transversal sub-manifold with a higher-dimensional ``segment" of the foliation. 

By the translation structure of torus given by $\C^n/\Lambda(g)$ we may also define the straight line flow on the torus to be the image of the straight line flow on $\C^n$ under the projection. More precisely, the straight line flow from $x \in X=\C^n/\Lambda(g)$ in the direction $v\in\C^n$ is defined as $f^v_t(x+\Lambda(g))=(x+tv)+\Lambda(g)$.

\begin{proof}[Proof of Theorem \ref{masurtorus}]
Suppose $Ra_tX$ is recurrent in the moduli space of K\"ahler structures, we hope to prove unique ergodicity of the horizontal foliation $\mathcal{F}^h$. 

Suppose a subsequence of $Ra_t X$ converges to $X_\infty$ in $\mathrm{Kah}(T^n)$ and take $h_{t_m}:g_{t_m}X\to X_\infty$ to be diffeomorphisms with local differentials of norm close to $1$ for $t_n$ large enough. 

The horizontal foliation is uniquely ergodic if and only if the straight line flows $f^v_t,t\in\R$ are uniquely ergodic, where $v$ is taken over all possible directions in the tangent space $T_pF$ of a leaf $F\in \mathcal{F}$. Note that since the foliation is linear, the tangent space does not depend on the points of the leaves.

We prove by contradiction. Suppose $\mu$ and $\nu$ are different ergodic measures of the straight line flow that give different transverse measures on a sub-manifold $Q_0$ transverse to $\mathcal{F}^h$. Take a generic point $x$ of $\mu$ and a generic point $y$ of $\nu$. For any basis $v_i,1\le i\le n$ of the space given by $\mathcal{F}^h$, and for a flow box $Q=\{\prod_{i=1}^nf^{v_i}_{t_i} x: 0\le t_i\le \delta,x\in Q_0 \}$, we have by the Birkhoff ergodic theorem $$\frac{1}{\prod T^i}\int_0^{T^1}...\int_0^{T^n} 1_Q(\prod(f^{v_i}_{t_i})(x))dt_1...dt_n\to \mu(Q)$$ and $$\frac{1}{\prod T^i}\int_0^{-T^1}...\int_0^{-T^n}1_Q(\prod(f^{v_i}_{t_i})(y))dt_1...dt_n\to \nu(Q)$$ as $T^i\to \infty$ for any $Q\subset X$. 

Take $R$ to be the $2n$-cube with $x_\infty$  and $y_\infty$ as diagonal vertices in $X_\infty$ whose edges are parallel to $v_i$ in $\mathcal{F}^h$ or parallel to $w_i$ in $\mathcal{F}^v$. 

Suppose the lengths of the edges  in $\mathcal{F}^h$ are $l^i$ in the direction $v_i$, and the edges in $\mathcal{F}^v$ are of length at most $l^-$ in the direction $w_i$. Then $h_{t_m}^{-1} R$ is a cube in $g_{t_m} X$. We take $R_m$ to be the cube in $g_{t_m}X$ with diagonal vertices $g_{t_m}x$ and $g_{t_m}y$ whose edges are of lengths $l^i_m$ and $l^-_m$ and directions $v_i$ and $v_-$. By our assumption, $R_m$ is close to $h_{t_m}^{-1} R$ and we have $l^i_m\to l^i$, $l^-_m\to l^-$, as $n\to \infty$. Now consider $g_{t_m}^{-1} R_m\subset X$, it will become a very thin cube with $x$ and $y$ as diagonal vertices and edges of lengths $T_n^i=e^{-t_m} l^i_m\to \infty$ and $e^{t_m}l^-_m\to 0$ as $t_m\to -\infty$. And we see that for a transverse sub-manifold $Q_0$ of $\mathcal{F}^h$ the $n$-face in $\mathcal{F}^h$ starting from $x$ and the $n$-face in $\mathcal{F}^h$ ending with $y$ intersects $Q$ approximately the same amount of times, unless $x$ is very close to the boundary $\partial Q_0$. Take $v_i$ so that the $n$-face of the cube in $\mathcal{F}^h$ is non-degenerate.

Now take $Q$ to be $Q=\{\prod_{i=1}^nf^{v_i}_{t_i} x: 0\le t_i\le \delta,x\in Q_0 \}$. Then we have $1_Q(\prod(f^{v_i}_{t_i})(x))=1_Q(\prod(f^{v_i}_{s_i})(y))$ for $t_i-s_i=T^i_m$ unless $\prod(f^{v_i}_{t_i})(x)\in Q_\epsilon$ where $Q_\epsilon=\{\prod_{i=1}^nf^{v_i}_{t_i} x: 0\le t_i\le \delta,x\in Q_0,dist(x,Q_0)<\epsilon\}$ and $\epsilon\ge ce^{-t_m}$. Here $c$ depends only on $v_-$ and $y$. Thus we have $$\mid \frac{1}{\prod T^i_m}\int_0^{T^1_m}...\int_0^{T^n_m} (1_Q(\prod(f^{v_i}_{t_i})(x))-1_Q(\prod(f^{v_i}_{t_i-T^i_m})(y)))dt_1...dt_n \mid $$$$\le \frac{1}{\prod T^i_m}\int_0^{T^1_m}...\int_0^{T^n_m} 1_{Q_\epsilon}(\prod(f^{v_i}_{t_i})(x))dt_1...dt_n\to \mu(Q_\epsilon)$$ as $n\to \infty$. Thus we get $|\mu(Q)-\nu(Q)|\le |\mu(Q_\epsilon)|$ for any distinct ergodic probability measures $\mu$ and $\nu$. Since $\mu$ supports the Lebesgue measure, we have $\mu(Q_\epsilon)\to 0$ as $\epsilon\to 0$. This contradicts to our assumption. Thus the horizontal foliation $\mathcal{F}^h$ can only have a unique transverse measure up to scale. This proves our theorem.
\end{proof}

Now we shall consider in greater details when exactly are foliations uniquely ergodic and when are the corresponding geodesic flow recurrent.

For complex tori, the unique ergodicity of irrational straight line flow is a classical result.
\begin{thm}\label{slftorus} The straight line flow on a complex tori is uniquely ergodic if and only if it is totally irrational, i.e. the direction of the straight line flow is not in the span of any lower-dimensional sub-lattice of $\Lambda(g)$.
\end{thm}

Thus for linear foliation we have the following corollary.
\begin{cor} \label{unierg}The horizontal foliation is either uniquely ergodic and thus minimal, or non-minimal where each leaf is rational in the sense that it is contained in the span of any lower-dimensional sub-lattice of $\Lambda(g)$.
\end{cor}
\begin{proof} If the horizontal foliation is contained in a subspace $V\subset \R^{2n}$ which is the span of a lower-dimensional sub-lattice of $\Lambda(g)$ then it cannot be uniquely ergodic: take $V^\prime$ is a subspace of $\R^{2n}$ spanned by another sub-lattice of $\Lambda(g)$ so that $V^\prime\oplus V=\R^{2n}$, then any measure on $V^\prime$ defines a transverse measure of the horizontal foliation.

However if the horizontal foliation is contained any rational subspace then there must be a direction in the foliation that is totally irrational, so by the unique ergodicity of totally irrational straight line flows the foliation is also uniquely ergodic.

\end{proof}

On the moduli space of flat K\"ahler metrics, however, there doe not exist such a clear distinction between divergent and non-divergent geodesic flows. However, we do know that geodesic flows generated by horizontal foliations are``almost always" ergodic by the following theorem.

\begin{thm}[See \cite{MooErg}] The right multiplication $Ra_t$ is ergodic on $\mathrm{SL}(2n,\Z)\backslash \mathrm{SL}(2n,\R)$ with respect to the Haar measure. 
\end{thm}
This implies that almost every geodesic flow is dense in the K\"ahler moduli space, hence recurrent.

The converse of Theorem \ref{masurtorus} is not true: there exists divergent geodesic flows that are generated by uniquely ergodic horizontal and vertical foliations. For example, if a leaf of the horizontal foliation contains a closed curve in the torus with non-trivial topology then the geodesic flow must be divergent in the moduli space of flat metrics. However, the foliation itself may still be uniquely ergodic as long as one direction in the foliation is totally irrational. One example is given as follows.\begin{prop}\label{exp}
Take $$g=\begin{pmatrix} 1 &0&0&0\\ 0&1&0&0\\ \sqrt{3} &\sqrt{5} & 1&\sqrt{2}  \\ 0&0&0 &1\end{pmatrix}$$ and $X=\C^2/\Lambda(g)$.
Then the horizontal foliation of $X=\C^2/\Lambda(g)$ is uniquely ergodic, but the geodesic flow $Ra_t (X,\omega_0)\in \mathrm{Kah}(\mathbb{T}^4)$ is divergent.
\end{prop}
\begin{proof}
The horizontal foliation is unique ergodic since $(0,0,1,0)$ is a totally irrational linear sum of the row vectors of $g$. Thus the straight line flow in this direction is uniquely ergodic.

However, the geodesic flow starting at $g$ is divergent since the length of the closed curve given by the vector $(1,0,0,0)$ tends to zero as $t\to \infty$. 
\end{proof}

\section{Preliminaries on Kummer Surface}\label{kum}
Next, we further exploit the form and property of the geodesic flow for Kummer surfaces in complex dimension 2. In this case, we may, by the Torelli theorem, explicitly calculate the geodesic flow and deduce some properties. First we introduce some preliminaries on Kummer surfaces and K3 surfaces, period map and the complex and generalized K\"ahler moduli spaces. 
\subsection{K3 surfaces and its moduli spaces}

Here we introduce the definition and basic properties and Kummer surface and K3 surface, including the Torelli theorem and the denseness of Kummer surface in the complex moduli space of K3.

\begin{defi} A \emph{K3 surface} is a complex surface $X$ which is simply connected and has a unique $($up to scale$)$ non-vanishing holomorphic $2$-form $\Omega_X$. 
\end{defi}

\begin{thm}[See Theorem 7.1.1 of \cite{HuyK3}] All K3 surfaces are diffeomorphic.  
\end{thm}
From now on we denote the differentiable manifold underlying a K3 surface by $M_0$. For any K3 surface $X$, we define a marking $f$ of $X$ to be a diffeomorphism $f: M_0\to X$. A marking $f$ identifies the homology space $H_2(X,\Z)$ with $H_2(M_0,\Z)$ through $f_*$.

\begin{defi}[Period Domain] Let $M_0$ be the differentiable manifold underlying K3 surfaces. The period domain of K3 surfaces is the set of Hodge structure on $H^2(M_0,\C)$, denoted as $\mathrm{Per}(M_0)=\{f^*(H^{2,0}(X,\C))\subset H^2(M_0,\C):f: M_0\to X \text{ a diffeomorphism}\} \subset \mathbb{P}H^2(M_0,\C)$.

The period map from the Teichm\"uller space of M to the period domain is given by $\mathrm{Per}: \mathrm{Teich} (M_0)\to \mathrm{Per}(M_0): [(f,X)]\to f^*(H^{2,0}(X,\C))$.
\end{defi}

The relation between the period domain, the Teichm\"uller space and the moduli space of complex structures of K3 surfaces is described by the Torelli theorem.
\begin{thm}[Torelli Theorem]
The period domain of K3 surfaces is given by $\mathrm{Per}(K3)=\{\alpha\in \mathbb{P}H^2(M_0,\C):\alpha\cdot\alpha=0,\alpha\cdot \bar{\alpha}>0\}. $ The period map $\mathrm{Per}$ is a local isomorphism. Two complex K3 surfaces $X$ and $X^\prime$ are bi-holomorphic if and only if there exists a Hodge isometry $H^2(X,\Z)\to H^2(X^\prime,\Z)$.
\end{thm}

Noting that mapping class group is generated by change of markings on the Teichm\"uller space, the Torelli theorem give us the following description of the moduli space. 
\begin{cor}The mapping class group of K3 surfaces is given by $\Gamma=O(H^2(M_0,\Z))$ which is a lattice in $O(H^2(M_0,\R))$. The moduli space of complex structures of K3 surfaces is $\mathrm{Mod}(K3)=\Gamma\backslash \mathrm{Per}(K3)$. 
\end{cor}

\begin{rmk} We have a bijection between Hodge structures and oriented positive 2-planes in the real cohomology space of $M_0$, given by $\Omega\in H^{2,0}(X,\C) \to span(\Re\Omega, \Im\Omega)\subset H^2(X,\R)$. Indeed, the condition $\Omega\cdot\Omega=0,\Omega\cdot\bar{\Omega}>0$ is equivalent to $\Re\Omega\cdot\Re\Omega=\Im\Omega\cdot\Im\Omega>0, \Re\Omega\cdot\Im\Omega=0$ which is equivalent to the two plane being positive and a complex scalar for $\Omega$ is simply rotation and real rescale of basis in the 2-plane. So we may identify the period domain with the Grassmannian of oriented positive two-planes in $H^2(M_0,\R)$. More precisely, this gives an isomorphism $\mathrm{Per}(K3)\simeq SO(3,19,\R)/(SO(2)\times SO(1,19))$. The complex moduli space $\Gamma\backslash \mathrm{Per}(K3)$ is highly non-Hausdorff, see \cite{Vererg} for a more detailed discussion on this.
\end{rmk}


If we add the K\"ahler structure into consideration, the K\"ahler class and the Hodge structure together determine a positive 3-plane in the real cohomology space together with a direction in it specifying the K\"ahler class. So we have a map from the moduli space of K\"ahler structures to the space of positive 3-Grassmannian with a specified unit vector up to change of marking. We write this as $\mathrm{Kah}(K3)\to \Gamma\backslash SO(3,19,\R)/(SO(2)\times SO(19))$. 

However, this map is not surjective since a K\"ahler class cannot have zero intersection with curves. One way to solve this is to allow for degenerate K\"ahler forms, i.e. non-degenerate symplectic 2-forms that are in the positive cone but on the boundaries of the K\"ahler cone. 

\begin{defi}[Degenerate K\"ahler Form \cite{Kobdeg}] We define a degenerate K\"ahler form on a K3 surface $X$ to be a K\"ahler-Einstein orbifold metric on $X$ which is given by a K\"ahler-Einstein form on $Y$ where $X\to Y$ is a minimal resolution of $Y$ and $Y$ is a compact complex surface with at most rational double points. We call such $Y$ a \emph{generalized K3 surface}.

We define the period domain of degenerate K\"ahler forms of K3 to be $\mathrm{KPer}=\{(\alpha,\omega)\in \mathbb{P}H^2(M_0,\C)\times \mathbb{P}H^2(M_0,\R): \alpha\in \mathrm{Per}(K3), \alpha\cdot\omega=0,\omega\cdot\omega>0\}$ this is equivalent to pairs $\{(V,\omega)\}$ where $V$ is taken over positive 3-Grassmannians in $ H^2(M_0,\R)$ and $\omega$ are taken over vectors in $V$. So we have the identification
$\mathrm{KPer}\simeq SO(3,19)/(SO(2)\times SO(19))$. And we take the period map as $(X,\omega)\mapsto (\mathrm{Per}(X),[\omega]) $ from the space of marked K3 surfaces with degenerate K\"ahler forms to the period domain.
\end{defi}

The following proposition shows that after adding in the degenerate K\"ahler forms, every positive 3-plane in the real cohomology space can indeed be realized by some K3 surface. 
\begin{prop}[Torelli Thoerem for Generalized K\"ahler Polarized K3 \cite{Kobdeg}] 
The period map $(X,\omega)\to (\mathrm{Per}(X),[\omega]) $ from marked K3 surfaces with degenerate K\"ahler forms to its period domain $\mathrm{KPer}$ is surjective. Moreover, if $F: H^2(X,\Z)\to H^2(X^\prime,\Z)$ is an isometry $($with respect to inner product given by intersection$)$ whose extension to $H^2(X,\C)\to H^2(X^\prime,\C)$ preserves the Hodge structure and the K\"ahler class, then there is a unique isomoprhism $f: X\to X^\prime$ such that $F=f^*$.
\end{prop}
Adding the change of markings into consideration, we have the following.
\begin{cor} The moduli space of generalized K\"ahler forms of K3 surfaces is given by $$\mathrm{\overline{Kah(M_0)}}\simeq \Gamma\backslash \mathrm{KPer}$$$$\simeq \Gamma\backslash SO(3,19)/(SO(2)\times SO(19)).$$
\end{cor}
This gives a homogeneous description of the moduli space of generalized K\"ahler froms of K3 surfaces. Note that this space is Hausdorff while the complex moduli space of K3 surfaces is not. Elements in the generalized period domain $\mathrm{KPer}$ leaves the compact sets exactly when the positive $3$-plane approaches the light cone in the real cohomology space. 

\subsection{Kummer Surfaces}
Kummer surfaces is a special class of K3 surfaces with a very concrete construction. 

\begin{defi}[Kummer surface]
Given a complex torus $A=\C^2/\Lambda(g)$ where $\Lambda(g)$ is the lattice in $\C^2$ generated by the row vectors of $g\in \mathrm{GL}(4,\R)$, we may consider the involution $\tau:(z,w)\to (-z,-w)$ on $\C^n$ which has $16$ fixed points in $A$. The Kummer minifold $\mathrm{Kum}(g)$ is defined as the minimal resolution of $A/\tau$ where $A=\C^2/\Lambda(g)$ is a complex torus given by $g\in \mathrm{GL}(4,\R)$. The minimal resolution is given by blowing up the $16$ fixed points of the involution $\tau$. 

We call $A/\tau$ the exceptional Kummer surface and we call the blow-up of the $16$ fixed points the exceptional curves of the Kummer surface $\mathrm{Kum}(g)$. We denote the exceptional curves by $C_i,1\le i\le 16$.
 \end{defi}
Indeed, we may check that the image of the holomorphic 2-form $dz\wedge dw$ extends to a non-vanishing holomorphic 2-form on the Kummer surface and that the Kummer surface is simply connected. Therefore it is a K3 surface.

Next we fix a marking on Kummer surfaces for simplicity for calculation.

For any $g\in \mathrm{SL}(4,\R)$, we identify its row vectors with complex vectors in $\C^2$ to obtain a lattice $\Lambda(g)$, and take a marked Kummer surface $(\phi_g, \mathrm{Kum}(g))$ in the following way. 
For the identity element $I\in \mathrm{SL}(4,\R)$ we fix a marking $\phi_I : \mathrm{Kum}(I)\to M_0$, and for any $g\in \mathrm{SL}(4,\R)$, we take $\phi_g=g_*\phi_I$ where we consider $g$ acting on the left on $\C^2$ as a homeomorphism $\C^2/\Lambda(I)\to \C^2/\Lambda(g)$ and $g_*$ is the corresponding homeomorphism between the Kummer surfaces. We denote by $C_i$ the image of the 16 exceptional curves of $\mathrm{Kum}(I)$ in $M_0$ under $\phi_I$, then by the way we defined $\phi_g$ they are also the image of the 16 exceptional curves of $\mathrm{Kum}(g)$ under $\phi_g$. We denote by $\omega_g$ the degenerate K\"ahler form on the Kummer surface $\mathrm{Kum}(g)$ given by the standard K\"ahler form $\omega_0$ on $\C^2$. Then under this special marking  $(\phi_g)^*(\omega_g)$ is constant on $M_0$. 

We define the Kummer lattice $K$ to be the smallest primitive sublattice of $H^2(X,\Z)$ that contains $C_i,1\le i\le 16$.

The map from the Teichm\"uller space of complex-2-tori to $\mathrm{Per}(M_0)$ given by $\C^2/g \to \phi_g^*(H^{2,0}(\mathrm{Kum}(g))$ has image $\mathrm{P(Kum)}=\{\alpha\in \mathrm{Per}(M_0): \alpha\cdot C_i=0, 1\le i\le 16\}$ by the following proposition and the Torelli thoerem. 

\begin{prop}[\cite{Nikkum}, see also Proposition 14.3.17 of \cite{HuyK3}  ]\label{kumcri} A K3 surface is Kummer if and only if there exists a primitive embedding of the Kummer lattice $K$ into $NS(X,\Z)=\{c\in H^2(X,\Z):c\cdot\Omega=0\}$. 
\end{prop}

So we see that Kummer surfaces with a certain marking make up a 5 dimensional closed subspace of the period domain. We call it the specially marked Kummer locus and denote it by $\mathrm{P(Kum)}$. However, considering the different markings, or equivalently the action of the mapping class group, the orbit of this closed subspace is in fact dense in the period domain.
\begin{prop}[Density theorem] Kummer surfaces are dense in the period domain.
\end{prop}
\begin{proof}[Sketch of proof, See \cite{HuyK3}, \cite{Kobdeg} for more detail] By Proposition \ref{kumcri} a K3 surface of maximal Picard rank is a Kummer surface if $T(X)=NS(X,\R)^\perp$ is a positive definite even oriented rank 2 lattice in $H^2(X,\Z)$. This happens if the 2-plane given by the period $[\Omega_X]$ of $X$ is a rationally defined 2-plane in $H^2(X,\R)$ such that $x\cdot x\in 4\Z$ if $x\in span(\Re \Omega,\Im \Omega)\cap H^2(X,\Z)$. Such 2-planes are dense in the Grassmannian of positive 2-planes in $H^2(X,\R)$.
\end{proof}

If we add the K\"ahler from into consideration, the flat metrics on tori maps to degenerate K\"ahler forms on the K3 surface, in fact, they give exactly the degenerate K\"ahler forms that have zero volume on the 16 exceptional curves.

\begin{prop} The image of $(\phi_g,\mathrm{Kum}(g),\phi_g^*(\omega_0))$ in $\mathrm{KPer}$ is $$\mathrm{K(Kum)}=\{(\alpha,\omega)\in \mathrm{KPer}: \alpha\cdot C_i=0, \omega\cdot C_i=0, 1\le i\le 16\}\subset \mathrm{KPer}.$$ We call this the degenerate flat Kummer locus. In particular, the image of the K\"ahler moduli space of complex torus into generalized K\"ahler moduli spcae of K3 surfaces is $\Gamma\backslash\Gamma(\mathrm{K(Kum)})$.
\end{prop}
\begin{proof} 
Fix an $\Omega\in \mathrm{P(Kum)}$ and the complex K3 surface $X=\mathrm{Kum}(g)$ with period $\Omega$, we only need to prove that every K\"ahler class in the K\"ahler cone that is orthogonal to $C_i$ can be realized as the image of the standard K\"ahler form $\omega_0$ on $\C^2$ for some $h\in g \mathrm{GL}(n,\C)$. Firstly, the image $(\Omega,\phi_g^*(\omega_0))$ is in $\mathrm{K(Kum)}$ since the degenerate K\"ahler form has zero volume on the 16 exceptional curves. And we note that the fiber of $\mathrm{K(Kum)}$ over $\Omega$, denoted $\mathrm{Kah}(\Omega)$ is one component (determined by orientation) of $\{\omega\in H^2(X,\R): \omega\cdot C_i=\omega\cdot \Omega=0, \omega\cdot\omega>0\}$ which is a connected open subset of a $4$ dimensional subspace in $H^2(X,\R)$.
The right multiplication of $\mathrm{GL}(n,\C)$ on $\mathrm{GL}(2n,\R)$ preserves the complex structure. Hence we fix the marking $\phi_g$ to descend this action to $\mathrm{Kah}(\Omega)$. The stablizer of $\omega_g$ is a conjugate of $\mathrm{U}(n)$. So we have a differentiable map $\mathrm{GL}(n,\C)/\mathrm{U}(n)\to \mathrm{Kah}(\Omega): h\to \omega_{gh}$ which is locally injective. Thus the image must be the whole set. 
\end{proof}

\section{Special Lagrangian Foliation and Geodesic Flow for Kummer Surfaces}\label{slag}

We define the linear foliations and the geodesic flow for Kummer surfaces to be the image of the linear foliations and the geodesic flow on the tori generating the Kummer surfaces. 

The horizontal and vertical foliations of a Kummer surface $X=\mathrm{Kum}(g)$ are defined for a specified $g\in \mathrm{GL}(4,\R)$ as the image of the horizontal and vertical foliation of the torus $X=\C^2/\Lambda(g)$ under the map $X\to X/\{\pm 1\}$. These are singular special Lagrangian foliations on the corresponding exceptional Kummer surfaces and they are defined outside the exceptional curves on the Kummer surfaces.

\begin{prop}\label{lagfol}
Suppose $(X,\omega)\in \mathrm{K(Kum)}$ is a Kummer surface given by $X=\mathrm{Kum}(g)$ with degenerate K\"ahler form $\omega$, then
the horizontal foliations of $\mathrm{Kum}(gh),\omega,h\in U(2)$ are in one to one correspondence to linear special Lagrangian foliations of $(X,\omega)$ outside the 16 exceptional curves. 
\end{prop}
\begin{proof} This is done by direct calculation.
For a linear subspace  $V=\{y_1=ax_1+bx_2, y_2=cx_1+dx_2\}$, $\omega_0$ and $\Im (dz_1\wedge dz_2)$ to vanish on $V$ if and only if $a=-d,b=c$, in which case $\Re (dz_1\wedge dz_2)\mid V>0$ for a suitable orientation. On the other hand the leaves of possible horizontal foliations on $X$ are the images $\{(x_1,c_1,x_2,c_2):x_1,x_2\in \R\}h$, where $h\in \mathrm{U}(n)$, for $\det((h^{ij})_{i,j=1,3})\neq 0$, we may write this as a subspace where $\begin{pmatrix} a&b\\c&d\end{pmatrix}$ is given by $\Im h (\Re h)^{-1}$ so we have $a=-d,b=c$. Furthermore, any $\begin{pmatrix} a&b\\b&-a\end{pmatrix}$ can always be written as $\Im h (\Re h)^{-1}$ for some $h\in \mathrm{U}(n)$.

For $V$ that cannot be written in the form $\{y_1=ax_1+bx_2, y_2=cx_1+dx_2\}$, we simply observe that we may apply a unitary transformation to bring it to this form.
\end{proof}

A geodesic flow of Kummer surfaces starting at $\mathrm{Kum}(g)$ (with direction given by $g$) is defined as $Ra_t(\mathrm{Kum}(g),\omega_g)=(\mathrm{Kum}(ga_t),\omega_{ga_t}) $, i.e. bringing the Kummer surface generated by $g\in \mathrm{GL}(4,\R)$ to the one generated by $ga_t$. This is obtained by projecting the geodesic flow in the moduli space of complex torus down to the Kummer locus $\mathrm{K(Kum)}$.

We may apply our results on complex tori to Kummer surfaces to get similar connection between geodesic flow and ergodicity of foliations for Kummer surfaces.

\begin{proof}[Proof of Theorem \ref{masurkum} for Kummer surfaces]If the vertical or horizontal foliations on a torus is uniquely ergodic, then its image on the exceptional Kummer surface must be uniquely ergodic. If the geodesic flow in generalized K\"ahler moduli space of K3 surfaces is recurrent in compact subsets, then the pre-image in the K\"ahler moduli space of tori must also be recurrent. Thus by Theorem \ref{masurtorus}, assume the image of $\{Ra_t X: t\le 0\}$ in $\overline{\mathrm{Kah}}(M_0)$ is recurrent. Then the horizontal foliation must be uniquely ergodic.
\end{proof}

By the Torelli theorem and the period map, we may explicitly calculate the form of the geodesic flow.
\begin{prop}
The geodesic flow $Ra_t X$ of a Kummer surface $X$ in the period domain of degenerate K\"ahler forms $\mathrm{KPer}$ is of the form $(e^{-2t}\eta_1+e^{2t}\eta_2+\sqrt{-1}\eta_3,\omega)$ where $\eta_i,\omega\in H^2(X,\R)$ satisfy $\eta_i\in H^2(X,\R), \eta_i\cdot C_j=0, \omega\cdot C_j=0, 1\le i\le 3,1\le j\le 16$ and $(\eta_i)^2=0, \eta_i\cdot\eta_3=0$ for  $i=1,2$ and $2\eta_1\cdot \eta_2=(\eta_3)^2>0$, $\omega\cdot \eta_i=0,i=1,2,3$. 

Thus in the K\"ahler moduli space $\mathrm{\overline{Kah}}(M_0)$, the geodesic flows are of the form $[e^{-2t}\eta_1+e^{2t}\eta_2+\sqrt{-1}\eta_3,\omega]=\Gamma (span(e^{-2t}\eta_1+e^{2t}\eta_2,\eta_3,\omega)) $$\in \mathrm{\overline{Kah}}(M_0)\simeq \Gamma \backslash SO(3,19,\R) /(SO(2)\times SO(19))$.

\end{prop}
\begin{proof}
Suppose $X=\mathrm{Kum}(g)$, first we calculate the periods of $Ra_t X$ using the markings $\phi_{ga_t}$ and the basis given as follows.

Suppose $g=\begin{pmatrix} \lambda^1 \\ \lambda^2\\ \lambda^3\\ \lambda^4\end{pmatrix}=\begin{pmatrix} \lambda^1_1 & \lambda^1_2 \\ \lambda^2_1 &\lambda^2_2 \\ \lambda_1^3 & \lambda_2^3\\ \lambda_1^4 & \lambda_2^4\end{pmatrix} $ where $\lambda^i_j\in \C$. Take $T_{ij}$ to be the torus in the Kummer surface given by  the image of the torus in $\C^2/\Lambda(g)$ spanned by $\lambda^i$ and $\lambda^j$. Then $T_{ij},C_k$ makes up a basis for $H_2(X,\Q)$
and $T_{ij}$ is a basis for $(\{C_j\})^\perp\subset H_2(X,\Z)$. 

Under this basis and the corresponding marking, we have $\mathrm{Per}(\mathrm{Kum}(g))=[\int_{T_{ij}}dz^1\wedge dz^2,\overbrace{0,...,0}^{16\text{ zeros }}]$. So by direct calculation, we have $$\mathrm{Per}( Ra_t(X))=\mathrm{Per}(\mathrm{Kum}\begin{pmatrix} 
e^{-t} \Re\lambda^1_1 & e^t \Im\lambda^1_1& e^{-t} \Re \lambda^1_2  & e^{t} \Im \lambda^1_2 \\ 
e^{-t} \Re\lambda^2_1 & e^{t} \Im\lambda^2_1 & e^{-t} \Re\lambda^2_2 & e^{t} \Im\lambda^2_2 \\ 
e^{-t} \Re\lambda_1^3 & e^{t} \Im\lambda_1^3 & e^{-t} \Re\lambda_2^3 & e^{t} \Im\lambda_2^3\\ 
e^{-t} \Re\lambda_1^4  & e^{t} \Im\lambda_1^4 & e^{-t} \Re\lambda_2^4& e^{t} \Im\lambda_2^4 \end{pmatrix})$$
$$=[(e^{-t} \Re\lambda^i_1+\sqrt{-1}e^{t}\Im \lambda^i_1)(e^{-t} \Re\lambda^j_2+\sqrt{-1}e^{t}\Im\lambda^j_2)$$$$-(e^{-t} \Re\lambda^i_2+\sqrt{-1}e^{t}\Im\lambda^i_2)(e^{-t} \Re\lambda^j_1+\sqrt{-1}e^{t}\Im\lambda^j_1),0,...,0]$$
$$=[e^{-2t}\eta_1^{ij}+e^{2t}\eta_2^{ij}+\sqrt{-1}\eta_3^{ij},0,...,0]$$

where $$\eta_1^{ij}=\Re \lambda^i_1\Re\lambda^j_2-\Re\lambda^i_2\Re\lambda^j_1,$$ $$\eta_2^{ij}=\Im \lambda^i_1\Im\lambda^j_2-\Im\lambda^i_2\Im\lambda^j_1,$$ $$\eta_3^{ij}=\Re \lambda^i_1\Im\lambda^j_2-\Re\lambda^i_2\Im\lambda^j_1+\Im \lambda^i_1\Re\lambda^j_2-\Im\lambda^i_2\Re\lambda^j_1.$$

For the K\"ahler form $\omega_g$, we have under these coordinates $\mathrm{Per}(\phi^*_{ga_t}(\omega_g))=[\int_{T_{ij}}dx_1\wedge dy_1+dx_2\wedge dy_2,\overbrace{0,...,0}^{16\text{ zeroes}}]$ where $\int_{T_{ij}}dx_1\wedge dy_1+dx_2\wedge dy_2$ on $\mathrm{Kum}(ga_t)$ is $\int_{T_{ij}}dx_1\wedge dy_1+dx_2\wedge dy_2=e^{-t} \Re\lambda^i_1 e^{t} \Im \lambda^j_1-e^{-t} \Re\lambda^j_1 e^{t} \Im \lambda^i_1+e^{-t} \Re\lambda^i_2 e^{t} \Im \lambda^j_2-e^{-t} \Re\lambda^j_2 e^{t} \Im \lambda^i_2= \Re\lambda^i_1  \Im \lambda^j_1- \Re\lambda^j_1  \Im \lambda^i_1+\Re\lambda^i_2  \Im \lambda^j_2- \Re\lambda^j_2  \Im \lambda^i_2$ which is invariant under $Ra_t$. Denote this 2-form as $\omega$ (which is equal to $[\omega_g]$).

So take $\eta_k=\sum\eta_k^{ij}T_{ij}^*$, we have $\mathrm{Per}(Ra_t(X))=e^{-2t}\eta_1+e^{2t}\eta_2+\sqrt{-1}\eta_3$. 

Next, we check the conditions they satisfy. We may check that $(\eta_i)^2=0, \eta_i\cdot \eta_3=0$ for  $i=1,2$ and $2\eta_1\cdot \eta_2=(\eta_3)^2>0$, $\omega\cdot \eta_i=0,i=1,2,3$ either by direct calculation, or by the fact that $(e^{2t}\eta_1+e^{-2t}\eta_2+\sqrt{-1}\eta_3,\omega)\in \mathrm{K(Kum)}$ for any $t$.

\end{proof}

Descending from the case of tori, for Kummer surfaces, the geodesic flow we defined acts on exceptional Kummer surfaces exactly by expanding the horizontal foliation and contracting the vertical foliations. Although we defined the geodesic flow as given by a ``direction" from $g\in \mathrm{SL}(2n,\R)$, we see in the calculation that the geodesic flow itself only depends on the directions of $\mathcal{F}^v$ and $\mathcal{F}^h$, and we may view $\eta_1$ and $\eta_2$ as invariants representing those foliations. 


We may also apply Corollary \ref{unierg} to get a necessary condition for horizontal foliations to be uniquely ergodic: it is uniquely ergodic only if $rank(\eta_2^\perp \cap H^2(X,\Z))< 19$.

\begin{proof}[Proof of Theorem 1.3]
The horizontal foliation $\mathcal{F}^h$ in the Kummer surface $X$ outside the exceptional curves is unique ergodic if and only if its pre-image in the torus $Y=\C^2/\Lambda(g)$ is uniquely ergodic, which happens if and only if its leaves are not contained in the real span of any sub-lattice of $g=\begin{pmatrix} \lambda^1 \\ \lambda^2\\ \lambda^3\\ \lambda^4\end{pmatrix}$. We next prove by contradiction. We show that if a leaf of the horizontal foliation is in the real span of a sub-lattice of $\Lambda$, then $rank(\eta_2^\perp \cap H^2(X,\Z))\ge 19.$

Assume that a leaf of the horizontal foliation is in the real span of a sub-lattice of $\Lambda$. 

We first assume that it is in the real span of $\lambda^1,\lambda^2,\lambda^3$. We claim that this happens if and only if $\Im \lambda^1,\Im \lambda^2,\Im \lambda^3$ are collinear over $\R$. Indeed, if $\Im \lambda^1,\Im \lambda^2,\Im \lambda^3$ are collinear, then $\Re \lambda^1,\Re \lambda^2,\Re \lambda^3$ cannot be collinear, so the 2-plane $\{y_1=y_2=0\}$ is in the real span of $\lambda^1,\lambda^2$ and $\lambda^3$. Conversely, if $\Im \lambda^1,\Im \lambda^2,\Im \lambda^3$ are not collinear, then $span(\lambda^1,\lambda^2,\lambda^3)\cap \{y_1=y_2=0\}$ must be 1-dimensional, which contradicts to a leaf of the horizontal foliation being in the real span $span(\lambda^1,\lambda^2,\lambda^3)$.

Now if we only have that a leaf of the horizontal foliation is in the real span of a sub-lattice of $\Lambda$, then we have $\Im \lambda^1,\Im \lambda^2,\Im \lambda^3$ are mutually linear after a left multiplication of $\mathrm{SL}(4,\Z)$ on $\Lambda$. So using the basis $T_{ik}, 1\le i<k\le 4, C_j, 1\le j\le 16 $ for $H^2(X,\Q)$, $\eta_2$ has 19 zero entries after a left multiplication of $SO(3,3,\Z)\times I_{16}$ on $\eta_2$. Thus we have  $rank(\eta_2^\perp \cap H^2(X,\Z))\ge 19$.
\end{proof}

We remark that the converse is not true. 
For example consider $$g=\begin{pmatrix} 1&1&0&0\\
1&\sqrt{2}&0 & 0\\ 0&-\sqrt{3}&1&0\\ 0&0&1&1\end{pmatrix}\text{ and } h=\begin{pmatrix} 1&0&0&0\\
1&1&0&0\\0&0&1&1\\0&\sqrt{3}&1&\sqrt{2}\end{pmatrix}.$$ Then the horizontal foliations on $\mathrm{Kum}(g)$ is uniquely ergodic while the horizontal foliations on $\mathrm{Kum}(h)$ is not. However, we may calculate that $\eta_2(g)$ and $\eta_2(h)$ differ only by an element in $O(H^2(M_0,\Z))$, so they are the same elements on the boundary of the complex moduli space. 
\bibliography{masurdiv}
\bibliographystyle{alpha}
\end{document}